\documentclass[11pt, leqno]{article}
\usepackage{amsmath,amsfonts,amssymb,wasysym}%,hyperref}
\usepackage[latin1]{inputenc}
\usepackage{shapepar}
\usepackage{graphicx}
\usepackage{ifpdf} %for .eps and .jpg
\usepackage{color}
\newcommand{\R}{{\mathbb R}}
\newcommand{\re}{{\mathbb R}}
\newcommand{\ren}{{\mathbb R}^N}

\newcommand{\be}[1]{\begin{equation}\label{#1}}
\newcommand{\ee}{\end{equation}}

\newcommand{\prf}{\par\smallskip\noindent{\sl Proof. \/}}
\newcommand{\finprf}{\unskip\null\hfill$\;\square$\vskip 0.3cm}

\newenvironment{proof}{\prf}{\finprf}
\newtheorem{theorem}{Theorem}[section]
\newtheorem{lemma}{Lemma}[section]

\newtheorem{proposition}[theorem]{Proposition}
\newtheorem{remark}[theorem]{Remark}
\newtheorem{definition}{Definition}[section]

\newcommand{\dint}{\displaystyle\int}
\numberwithin{equation}{section}

\def\qed{\,\unskip\kern 6pt \penalty 500
\raise -2pt\hbox{\vrule \vbox to8pt{\hrule width 6pt
\vfill\hrule}\vrule}\par}
\definecolor{darkblue}{rgb}{0.05, .05, .65}
\definecolor{darkgreen}{rgb}{0.1, .65, .1}
\definecolor{darkred}{rgb}{0.8,0,0}
\begin{document}
\title{\textbf{ Symmetrization for fractional elliptic \\ and parabolic equations and \\ an isoperimetric application}\\[7mm]}

\author{\Large  Yannick Sire\footnote{Universit\'e Aix-Marseille,  I2M, Centre de Math\'ematique et Informatique, Technop\^ole de Chateau-Gombert, Marseille, France, \
E-mail:~{\tt yannick.sire@univ-amu.fr}}\,,
\quad Juan Luis V\'azquez
\footnote{Departamento de Matem\'aticas, Universidad Aut\'onoma de Madrid, 28049 Madrid, Spain, \
E-mail: {\tt juanluis.vazquez@uam.es}}\,,
\\[8pt] \Large  and \ Bruno Volzone\footnote{Dipartimento di Ingegneria, Università degli Studi di
Napoli ``Parthenope'', 80143  Italia. \    E-mail: {\tt bruno.volzone@uniparthenope.it}}
 }

\date{} %%  this cancels date in article format

\maketitle

\begin{abstract}
We develop  further  the theory of symmetrization of fractional Laplacian operators contained in recent works of two of the authors. The theory leads to optimal estimates in
the form of concentration comparison inequalities for both  elliptic and parabolic equations.
In this paper we extend the theory for the so-called \emph{restricted} fractional
Laplacian defined on a bounded domain $\Omega$ of $\ren$ with zero Dirichlet conditions outside of $\Omega$. As an application, we derive an original proof of the corresponding fractional Faber-Krahn inequality. We also provide a more classical variational proof of the inequality.
\end{abstract}

\

\setcounter{page}{1}
%%%%%%%%%%%%%%%%%%%%%%%%%%%%%%%%%%%%%%%%%%%%%%%%%%%%%%%%%%%%%%%%%
\section{Introduction}\label{sec.intro}

In this paper we will develop further the theory of symmetrization for fractional Laplacian operators
initiated in the papers \cite{dBVol, VazVol1, VazVol2}, both in the elliptic and the parabolic setting, by extending it to a natural  version of the fractional Laplacian defined on a bounded domain $\Omega$ of $\ren$ which is known as the restricted fractional Laplacian. This research direction combines classical themes in the study of nonlinear elliptic and parabolic equations, like symmetrization and accretive operators,  with the recent interest in nonlocal versions of the diffusion operators, specially the fractional Laplacians. As an application of the obtained comparison results, we derive an original proof of the Faber-Krahn inequality (FKI) for such operators defined on the bounded domain $\Omega$.

Before entering into the description of our results, we review in this introduction the necessary information about symmetrization, the elliptic-to-parabolic
technique used to generate evolution semigroups, the precise definition of the fractional Laplacian operators and the relation among these topics. This constitutes a sort of review part of this paper.

\medskip

\noindent {\sc Definitions of fractional Laplacians on bounded domains.} When working in the whole space domain $\ren$ there are several equivalent definitions
of the fractional Laplacian operator $(-\Delta)^{\sigma/2}$, $0<\sigma<2$, classical references being \cite{Landkof72, Stein70}. The interest in these operators
has a long history in Probability since the fractional Laplacian operators of the form $(-\Delta)^{\sigma/2}$  are infinitesimal generators of stable L\'{e}vy
processes, see \cite{Applebaum, Bertoin, Valdinoc}. Further motivation and references on the literature are given for instance in \cite{BV2012, VazVol1}. A particular definition that has been convenient for symmetrization purposes defines the operator for every given $0<\sigma<2$ as the trace of a suitable Dirichlet-Neumann problem via an extended potential function $w$ that solves an elliptic equation in a upper half-space in ${\mathcal{H}}^+=\re^{N}\times(0,\infty)\subset
\re^{N+1}$. This is usually called Caffarelli- Silvestre extension \cite{CaffSilv}. It allows to reduce nonlocal problems involving $(-\Delta)^{\sigma/2}$ to suitable local problems (actually, a degenerate-singular elliptic equation),  defined in one more space dimension.

When we work on a bounded domain $\Omega\subset\ren$,  things complicate because there are several options for defining the fractional Laplacian operator $(-\Delta)^{\sigma/2}$. Two of
them appear often in the recent literature. In our previous work  \cite{VazVol2} on symmetrization  we have followed one of these approaches to define the
fractional Laplacian as the Dirichlet-to-Neumann map, through an extended potential function defined in a cylinder  ${\cal C}=\Omega \times (0,\infty)\subset
\re^{N+1}$,
as was proposed in  \cite{CT} and \cite{Colorado}. Zero values are assigned on the lateral boundary of $\cal C$. We call this operator the {\sl spectral
version}
of the fractional Laplacian on $\Omega$.  Let us call this operator ${\cal L}_1$ ($\cal L$ stands for the Laplacian). This setting allowed us to derive in
\cite{VazVol2} the desired symmetrization results, which extend the standard
symmetrization theory applied to elliptic and parabolic equations driven by the standard Laplace operator. But let us recall that there are remarkable
restrictions on their validity in the form of conditions on the nonlinearities that are allowed in the equations.

In this paper we will take the second usual approach to define  $(-\Delta)^{\sigma/2}$, which seems to be more natural in many applications. It consists in
keeping the definition of fractional
Laplacian in $\ren$ but asking it to act on the null-extensions to $\ren$ of functions $u(x)$ defined in $\Omega$. So in principle we can use the most common
formulation with a hyper-singular kernel
\begin{equation}
(-\Delta)^{\sigma/2}f(x)=c(N,\sigma)\int_{\ren} \frac{f(x)-f(y)}{|x-y|^{N+\sigma}}\,dy\,,
\end{equation}
on the condition that $f(y)=0$ for $y\not\in \Omega$. Let us call this operator ${\cal L}_2$. This option has been called the {\sl restricted Laplacian} on a
bounded domain \cite{SerVal,
BonfSireVaz2014},  but we will prefer the name {\sl natural fractional Laplacian} with Dirichlet conditions in this paper. The discussion on the relations and
differences between the two types of operators on bounded domains is currently being investigated by several authors. Thus, Musina and Nazarov
\cite{MusinaNazarov} use the name \sl fractional Laplacian with Navier  conditions \rm for the spectral version, and \sl fractional Laplacian with Dirichlet
conditions \rm for the restricted version.

Here, we want to extend to operator ${\cal L}_2$  the symmetrization theory we had developed for ${\cal L}_1$ in
the papers \cite{dBVol} and  \cite{VazVol1}. This has an independent interest since there are subtle differences between the two operators, see
\cite{BonfSireVaz2014, BV2015}.

 \medskip

\noindent {\sc Symmetrization.} Symmetrization is a very ancient geometrical idea that is used nowadays as an efficient tool of obtaining a priori estimates for
the solutions of different partial differential  equations,  notably  those of elliptic and parabolic type. Since the topic is so well known let us only recall
some facts that are relevant here. Symmetrization techniques appear in classical works like \cite{MR0046395, PS1951}. The application of  Schwarz  symmetrization
to obtaining a priori estimates for elliptic problems is already described in \cite{Wein62} and  \cite{Maz}.  The  standard  elliptic result refers to the
solutions of an equation of the form
$$
Lu=f,  \qquad Lu=-\sum_{i,j} \partial_i(a_{ij}\partial_j u)\,,
$$
posed in a bounded domain $\Omega\subseteq \ren$;  the coefficients $\{a_{ij}\}$ are assumed to be bounded, measurable and satisfy the usual ellipticity
condition; finally, we take zero Dirichlet boundary conditions on the boundary $\partial\Omega$. The  classical analysis introduced by Talenti \cite{Talenti1,
Talenti3} leads to pointwise comparison between the symmetrized version (more precisely the spherical decreasing rearrangement)  of the actual solution of the
problem $u(x)$ and the radially symmetric solution $v(|x|)$ of some radially symmetric model problem which is posed in a ball with the same volume as $\Omega$.
Sharp a priori estimates for the solutions are then derived. Extensions of this method to more general problems or related equations have led to a copious
literature.

\medskip

\noindent {\sc Elliptic approach to parabolic problems.} For parabolic problems this pointwise comparison fails and the appropriate concept is comparison of
concentrations, cf. Bandle \cite{Bandle, Band2} and Vázquez \cite{Vsym82}. The latter considers the evolution problems of the form
\begin{equation}\label{evol.pbm}
\partial_t u=\Delta A(u), \quad u(0)=u_0,
 \end{equation}
where $A$ a monotone increasing real function and $u_0$ is a suitably given initial datum which is assumed to be integrable. For simplicity the problem was
posed
for $x\in \ren$,  but bounded open sets can be used as spatial domains. The
novel idea of the paper was to use the famous Crandall-Liggett Implicit Discretization theorem \cite{CL71} to reduce the evolution problem to a sequence of
nonlinear elliptic problems of the iterative form
\begin{equation}
- h\,\Delta A(u(t_k))+u(t_k)=u(t_{k-1}),\quad k=1,2, \cdots,
\end{equation}
where $t_k=kh$, and $h>0$  is the time step per iteration. Writing $A(u)=v$, the resulting chain of elliptic problems can be written in the common form
\begin{equation}\label{ell.eq}
h \,L v + B(v)= f\,, \quad B=A^{-1}.
\end{equation}
General theory of these equations, cf.  \cite{BBC1975}, ensures that the solution map: $$T:f\mapsto u=B(v)$$ is a contraction in some Banach space, which
happens
to be $L^1(\Omega)$. Note that the constant $h>0$ is not essential, it can be put to 1 by scaling. In that context, the symmetrization
result can be split into two results:

(i) the first one applies to rearranged right-hand sides and solutions.  It says that
if two r.h.s. functions $f_1, f_2$, are rearranged and satisfy a concentration comparison of the form $f_1 \prec f_2$, then the same applies to the solutions,
in
the form $B(v_1)\prec B(v_2)$.\footnote{For the definition of  the order relation $\prec$, see Section \ref{sec.prelim}.}

(ii) The second result  aims at comparing the solution $v$ of equation \eqref{ell.eq} with  a non-rearranged function $f$  with the solution $\tilde v$
corresponding to $f^\#$, the  radially decreasing rearrangement of $f$. We obtain that  $\tilde v$ is a rearranged function and $ B(v^\#)\prec
B(\tilde v)$, i.\,e., $B(v)$ is less concentrated than $B(\tilde v)$.

This precise pair of comparison results can be combined to obtain similar results along the whole chain of iterations $u(t_k)$ of the evolution process, if
discretized as indicated above. This allows in turn to conclude the symmetrization theorems (concentration comparison and comparison of $L^p$ norms)
for the evolution problem \eqref{evol.pbm}. This approach can be used in many different situations. In particular, it will be used below.

\medskip

\noindent {\sc Symmetrization for equations with fractional operators.} The study of elliptic and parabolic equations involving nonlocal operators, usually of fractional type, is currently the subject of great attention. Symmetrization techniques were first applied to PDEs involving fractional Laplacian operators in the paper \cite{dBVol}, where the linear elliptic case is studied:
\begin{equation}
(-\Delta)^{\sigma/2}v=f,
\end{equation}
The paper uses an interesting technique of Steiner symmetrization of the extended problem, based on the Caffarelli-Silvestre extension for the definition of $\sigma$-Laplacian operator.
 In \cite{VazVol1} the last two authors of the present paper \normalcolor were able to improve on that progress and combine it with the  parabolic ideas of \cite{Vsym82} to establish
the relevant comparison theorems based on symmetrization for linear and  nonlinear parabolic equations.  To be specific, they dealt with equations of the form
\begin{equation}\label{nolin.parab}
\partial_t u +(-\Delta)^{\sigma/2}A(u)=f, \qquad 0<\sigma<2\,.
\end{equation}
Following the known theory for the standard Laplacian, the nonlinearity $A$ is an increasing real function such that $A(0)=0$, and we accept some extra
regularity conditions as needed, like $A$ smooth with $A'(u)>0$ for all $u>0$. The problem was posed in the whole space $\ren$. Special attention was
paid to  cases of the form $A(u)=u^m$ with $m>0$; the equation is then called the Fractional Heat Equation (FHE) when $m=1$, the Fractional Porous Medium
Equation (FPME)  if $m>1$, and the Fractional Fast Diffusion Equation (FFDE) if $m<1$. Let us recall that the linear equation \ $\partial_t u
+(-\Delta)^{\sigma/2}u=0$  is a model of so-called anomalous diffusion, a much studied topic in
physics.

The results of \cite{dBVol} and    \cite{VazVol1} include a comparison of concentrations, in the form \ $v^\#\prec \tilde v $, that parallels the result that
holds in the standard Laplacian case; note however that no pointwise comparison is obtained, so the result looks a bit like the parabolic results of the standard
theory  mentioned above.
Paper \cite{VazVol1} considers both problems posed in the whole space and on a bounded domain. In the latter case the spectral fractional Laplacian is always
chosen.

%%%%%%%%%%%%%%%%%%%%%%%%%%%%%%%%%%%%%%%%%%%%%%%%%%%%%%%%%%%
\section{Outline of results of the present paper}\label{sec.outline}

We are interested in considering the application of such symmetrization techniques to linear or nonlinear elliptic and parabolic equations with  fractional
Laplacian
operators posed on a bounded domain, when the natural (i.\,e., restricted) version of fractional Laplacian is used. We denote the operator by ${\cal L}_2$.

\medskip

\noindent {\sc Parabolic equations.} To be specific, we want to treat evolution equations of the  form
\begin{equation}\label{nolin.parab}
\partial_t u +(-\Delta)^{\sigma/2}A(u)=f, \qquad 0<\sigma<2\,.
\end{equation}
We want to consider as nonlinearity $A$ an increasing real function such that $A(0)=0$, and we may accept some other regularity conditions as needed, like $A$
smooth with $A'(u)>0$ for all $u>0$. The problem is posed in $\Omega$, a bounded subset of $\ren$ with smooth boundary.  The parabolic result is developed in
Section \ref{sec.par} and has to be compared with the results of papers \cite{VazVol1, VazVol2}. We will focus on the linear case  $A(u)=cu$. This is the case
that is needed in the isoperimetric application  that we study in Section \ref{sect.fk}.

\medskip

\noindent {\sc Elliptic equations.} The application of the method of implicit time discretization leads to the nonlinear equation of elliptic type
\begin{equation}\label{nolin.ell}
h\,(-\Delta)^{\sigma/2}v+B(v) =f
\end{equation}
posed again in the whole space $\Omega  \subset \ren$ with zero Dirichlet boundary conditions; $h>0$ is a non-essential constant, and the nonlinearity $B$ is
the
inverse function to the monotone function $A$ that appears in the parabolic equation \eqref{nolin.parab}. The elliptic results are developed in Sections
\ref{sect.ell} and
\ref{sec.ell2} and have to be compared with the results of papers \cite{dBVol, VazVol1} where the equation is posed either in $\ren$ or in $\Omega$ with operator
${\cal L}_1$. Note that the elliptic results we get cover the standard linear case where the term $B(v)$ disappears and we set $h=1$.

\medskip

\noindent {\sc A geometrical application. The Faber-Krahn inequality.} As an application, we  will use the symmetrization results to prove the Faber Krahn
inequality for the  fractional Laplacian operator on a bounded domain in both versions considered above.  We recall that the FKI is a classical eigenvalue
inequality, due separately to Faber \cite{F} and Krahn \cite{K}, based on a conjecture by
Rayleigh in 1877, that can be stated as follows:

{\sl Let $\Omega$ be a bounded domain in $\ren$ and let
$B$ be the ball centered at the origin with $Vol(\Omega) = Vol(B)$. Let $\lambda_1(\Omega)$
be the first eigenvalue of the Laplacian operator, with zero Dirichlet boundary conditions.
Then $\lambda_1(\Omega)\ge \lambda_1(B)$, with equality if and only if $\Omega=B$
almost everywhere. }

This is a classical result in the Calculus of Variations and proofs can be found in the classical books like Chavel's \cite{Chavel}, see a recent proof in
\cite{BF2012}. The question we want to address here is: will the result  also hold for the usual versions of the fractional Laplacian operator
$(-\Delta)^{\sigma/2}$ defined on bounded domains of $\ren$ with zero Dirichlet boundary conditions?

The answer is immediate in the case of the so-called {\sl
spectral version of the Dirichlet fractional Laplacian}, ${\cal L}_1$, since its eigenvalues, $\lambda_k(\mathcal{L}_{1}; \Omega)$,  are directly related to
those of the
standard Laplacian, $\lambda_k((-\Delta);\Omega)$, by the formula:
\begin{equation}\label{spect.se.sigma}
\lambda_k(\mathcal{L}_{1}; \Omega)=(\lambda_k((-\Delta);\Omega))^{\sigma/2}.
\end{equation}
However, no simple relation like this one happens for the natural fractional Laplacian with the definition restricted type, ${\cal L}_2$. In Section
\ref{sect.fk} we  will use our comparison results to present an original derivation of the fractional FKI. It does not make use of any variational
interpretation, but only of some properties of the evolution process. The FKI can also be studied either by probabilistic or variational methods. For
completeness, we also present a variational derivation, see more details in the mentioned section.

\medskip

\noindent {\sc Preliminary material and notation.} In the paper we will use standard concepts and notation on symmetrization as fixed in \cite{VazVol1}. We
gather the main facts that we did not present here in the first appendix for the  reader's convenience.

%%%%%%%%%%%%%%%%%%%%%%%%%%%%%%%%%%%%%%%%%%%%
\section{Elliptic Problem with lower-order term}\label{sect.ell}
\setcounter{equation}{0}

The case of the natural fractional Laplacian ${\cal L}_2$ will occupy our attention in this paper. We start our analysis by the following nonlocal elliptic
problem with Dirichlet condition:
\begin{equation} \label{eq.1}
\left\{
\begin{array}
[c]{lll}%
\left(  -\Delta\right)^{\sigma/2}v+  B(v)=f\left(  x\right)   &  & in\text{ }%
\Omega,\\[6pt]
v=0 &  & in\text{ }\R^{N}\setminus\Omega,
\end{array}
\right. %
\end{equation}
\normalcolor
where $\Omega$ is an open bounded set of ${\mathbb{R}}^{N}$,  $\sigma\in(0,2)$ and $f$ is an integrable function defined in $\Omega$. We are interested in
treating the case of a bounded domain $\Omega$ with the natural version of the fractional Laplacian. Exceptionally, $\Omega$ may be ${\mathbb{R}}^{N}$, but this
case was treated in \cite{VazVol1}. We assume that the nonlinearity is given by a function $B:\R_{+}\rightarrow\R_{+}$ which is  smooth and monotone increasing
with $B(0)=0$ and $B'(v) >0$. It is
not essential to consider negative values for our main results, but the general theory can be done in that greater generality, just by assuming that $B$ is
extended to a function $B:\R_{-}\rightarrow\R_{-}$ by
symmetry, $B(-v)=-B(v)$.  Note that we have changed a bit the notation with respect to equation \eqref{nolin.ell} in the introduction, by eliminating the
constant $h>0$, but the change is inessential for the comparison results.

For simplicity, all the restrictions of the fractional Laplacian operator will be denoted by
$(-\Delta)^{\sigma/2}$. The underlying assumption is that such operator will be restricted to the ground domain of each boundary value problem where it is
involved.

\medskip

\noindent {\sc The extension method.}  The Caffarelli-Silvestre method can be kept as extension to the whole ${\mathcal{H}}^+$, with the following important
proviso: the extension must act on the null-extensions to $\ren$ of functions
defined in $\Omega$. In view of this discussion, a solution to problem (\ref{nolin.ell}) is defined here as the trace of a properly defined Dirichlet-Neumann
problem in the following way:
\begin{equation}\label{eq.3}%
\left\{
\begin{array}
[c]{lll}%
-\operatorname{div}_{x,y}\left(  y^{1-\sigma}\nabla w\right)  =0 &  & in\text{
}{\mathcal{H^+}},\\[6pt]
\ w(x,0)=0 &  & \mbox{ for } x\in\R^{N}\setminus \Omega,\\[6pt]
\displaystyle{-\frac{1}{\kappa_{\sigma}}\lim_{y\rightarrow0^{+}}y^{1-\sigma}\,\dfrac{\partial w}{\partial y}(x,y)}+\,B(w(x,0))=f\left(  x\right)   &
&
\mbox{ for } x\in \Omega,
\end{array}
\right.
\end{equation}
where ${\mathcal{H}^+}:=\ren\times\left(  0,+\infty\right)  $, and $\kappa_{\sigma}$ is the constant
$\kappa_{\sigma}:=2^{1-\sigma}\,\Gamma(1-\frac{\sigma}{2})/{\Gamma(\frac{\sigma}{2})}$, see \cite{CaffSilv}, but such a value is not important here.

\subsection{Review of existence, uniqueness and main properties}

If problem \eqref{eq.3} is
solved in an appropriate sense, then the trace of $w$ over $\Omega$, $\text{Tr}_{\Omega}(w)=w(\cdot,0)=:v$ is said to be a solution to
problem (\ref{eq.1}). Note that the trace of $w$ on the bottom hyperplane  $\{y=0\}$ is the extension $\widetilde v$ of the function $v$ defined in $\Omega$ by
assigning the value zero outside of $\Omega$. This is what makes the difference with the case of the spectral Laplacian, where on the contrary the domain of the
extended function $w$ is the  cylinder ${\cal C}=\Omega\times(0,\infty)$ and $w$ takes zero boundary conditions on the lateral boundary,
$\Sigma=\partial\Omega\times [0,\infty)$, see \cite{VazVol1}. In accordance with our choice of operator and in view of the iteration process that leads to the
solution of the parabolic equations, we need only consider functions $f$ that are restrictions to $\Omega$ of functions defined in the whole of $\ren$.

In order to  make this more precise, we introduce the concept of weak solution to problem \eqref{eq.3}. It is convenient to define the weighted energy space
$$
X^{\sigma/2}(\mathcal{H}^{+})=\left\{  w\in H^{1}_{loc}(\mathcal{H}^+) : \int_{\mathcal{H}^+}
y^{1-\sigma}|\nabla_{x,y} w(x,y)|^{2}\,dxdy< \infty\right\}\,,
$$
equipped with the norm
\begin{equation}
\Vert w\Vert_{X^{\sigma/2}(\mathcal{H}^{+})}:=\left(  \int_{\mathcal{H}^+} y^{1-\sigma}\,|\nabla w(x,y)|^{2}\,dxdy\right)  ^{1/2}.\label{norm}
\end{equation}
For an open  set $E$ of $\R^{N}$ we denote by $H^{\sigma/2}(E)$ the classical fractional Sobolev space of order $\sigma/2$ over $E$. We recall that for any
$u\in
H^{\sigma/2}(\R^{N})$  there exists a unique $\sigma$-harmonic extension $w\in X^{\sigma/2}(\mathcal{H}^{+})$ of $u$ to the half space $\mathcal{H}^{+}$, namely
$w$ solves
\begin{equation}
\left\{
\begin{array}
[c]{lll}%
-\operatorname{div}_{x,y}\left(  y^{1-\sigma}\nabla w\right)  =0 &  & in\text{
}{\mathcal{H^+}},\\[6pt]
\ w(x,0)=u(x) &  & \mbox{ for } x\in\R^{N}.\\[6pt]
\end{array}
\right.  \label{extension}%
\end{equation}
Then we will write $w=Ext_{\mathcal{H}^{+}}(u).$ Moreover, for suitable functions $u$ we have
\[
(-\Delta)^{\sigma/2}u=-\frac{1}{\kappa_{\sigma}}\lim_{y\rightarrow0^{+}}y^{1-\sigma}\,\dfrac{\partial w}{\partial y}(x,y).
\]

Now, in order to give a proper meaning of solution to a problem like \eqref{eq.3} in a bounded domain $\Omega$, we define the space of all functions in
$X^{\sigma/2}(\mathcal{H}^{+})$ whose traces over $\R^{N}$ vanish outside of $\Omega$, namely
\begin{equation}
X_{\Omega}^{\sigma/2}(\mathcal{H}^{+})=\left\{  w\in X^{\sigma/2}(\mathcal{H}^{+}) :
w|_{\R^{N}\times\left\{0\right\}}\equiv0\,\,in\,\,\R^{N}\setminus\Omega\right\}.\label{spaceforext}
\end{equation}
The domain of the natural fractional Laplacian $(-\Delta)^{\sigma/2}$ is the space ${\cal H}(\Omega)$  defined by
\begin{equation}
{\cal H}(\Omega)=
\left\{
\begin{array}
[c]{lll}%
H^{\sigma/2}(\Omega) &  & if\,0<\sigma<1,\\[6pt]
H_{00}^{1/2}(\Omega) &  & if\,\sigma=1\\[6pt]
H_{0}^{\sigma/2}(\Omega) &  & if\,1<\sigma\leq2,
\end{array}
\right.  \label{spaceH}%
\end{equation}
where $H^{\sigma/2}(\Omega)$ and $H^{\sigma/2}_0(\Omega)$ are usual fractional Sobolev spaces, see  \cite{LiMa1972}, and
$$H_{00}^{1/2}(\Omega)=\left\{u\in H^{1/2}(\Omega):\int_{\Omega}\frac{u^{2}(x)}{d^2(x)}dx<\infty\right\}$$
with $d(x)=dist(x,\partial\Omega)$.\, It turns out that
\[
{\cal H}(\Omega)
=\left\{w|_{\Omega\times\left\{0\right\}}:w\in X_{\Omega}^{\sigma/2}(\mathcal{H}^{+})\right\}
\]
\normalcolor
(see \cite{BonfSireVaz2014} for a detailed account on this question). Then we provide the following definition:
\begin{definition}\label{definition2}
Let $\Omega$ be an open bounded set of $\ren$ and $f\in L^{1}(\Omega)$. We say that $w\in X_{\Omega}^{\sigma/2}(\mathcal{H}^{+})$ is a weak solution to
\eqref{eq.3}
if \ $Tr_{\Omega}(B(w))=:B(w(x,0))\in L^{1}(\Omega)$
and
\begin{equation}
\int_{\mathcal{H}^+}y^{1-\sigma}\nabla_{x,y} w\cdot \nabla_{x,y}\varphi\,dx\,dy+ \kappa_{\sigma}\int_{\Omega} B(w(x,0)) \,\varphi(x,0)dx
=\kappa_{\sigma}\int_{\Omega}f(x)\,\varphi(x,0)dx\label{weakfor}
\end{equation}
for all the test functions $\varphi\in C^{1}(\overline{\mathcal{H}^{+}})$ such that
$\text{Tr}_{\R^{N}}(\varphi)\equiv0\,\,in\,\,\R^{N}\setminus\Omega$.
\end{definition}
If $w$ is a solution to the ``extended problem'' \eqref{eq.3}, then the  trace function $v=\text{Tr}_{\Omega}(w)$ will be called a weak solution
to
problem
\eqref{eq.1}.
\begin{remark}
It is clear that if $B(t)=ct$ for all $t\geq0$ for some $c\geq0$, then problem \eqref{eq.3} becomes linear and the function $v=\text{Tr}_{\Omega}(w)$ belongs to
the space $\mathcal{H}(\Omega)$.
\end{remark}
Concerning existence of solutions, their smoothness and $L^{1}$ contraction properties, we excerpt some known results from \cite{BonfSireVaz2014,pqrv,pqrv2}
which can be extended for our more general nonlinearity $B$. For the regularity the reader may consult \cite{CabSire, CaffSilv, SerVal2, Silv1, Silv2}.

%%%%%%%%%%%%%%%%%%%%%%%%%%%%%%%%%%%%%%%%%
\begin{theorem}\label{th.exist}
For any $f\in L^{\infty}(\Omega)$ there exists a unique weak solution $w\in X_{0}^{\sigma/2}(\mathcal{C}_{\Omega})$ to problem \eqref{eq.3}, such that
$\text{Tr}_{\Omega}(B(w))\in L^{\infty}(\Omega)$. Moreover,

\smallskip

\noindent {\rm (i)} Regularity: we have \normalcolor  $w\in C^{\alpha}(\mathcal{C}_{\Omega})$ for every $\alpha<\sigma$ if $\sigma\le 1$ (resp. $w\in
C^{1,\alpha}(\mathcal{C}_{\Omega})$ for every $\alpha<\sigma-1$ if $\sigma> 1$). Arguing as in {\rm \cite{CT}}, higher regularity of $w$ depends easily on
higher
regularity of $f$ and $B$.

\noindent {\rm (ii)} $L^{1}$ contraction: if \, $w,\widetilde{w}$  are the solutions to \eqref{eq.3} corresponding to data $f,\widetilde{f}$, the
following
$L^{1}$
contraction property holds:
\begin{equation}
\int_{\Omega}\left[B(w(x,0))-B(\widetilde{w}(x,0))\right]_{+}dx\leq\int_{\Omega}[f(x)-\widetilde{f}(x)]_{+}dx.\label{contraction}
\end{equation}
In particular, we have that $w\geq0$ in $\overline{\mathcal{C}}_{\Omega}$ whenever $f\geq0$ on $\Omega$. Furthermore, if we put
$u:=B(w(\cdot,0))$,
then for
all
$p\in [1, \infty]$
we have
\[
\|u\|_{L^{p}(\Omega)}\leq \|f\|_{L^{p}(\Omega)}.
\]

\noindent {\rm (iii)} For data $f\in L^1(\Omega)$ the weak solution is obtained as the limit of the solutions of approximate problems with $f_n\in
L^1(\Omega)\cap L^\infty(\Omega)$, $f_n\to f$ in $L^1$, since then the sequence $\{B(w_n(x,0))\}_n$ also converges in $L^1$ to some $B(w(x,0)$,  and
$\|B(w(x,0))\|_1\le \|f\|_1$, hence $v_n$ in uniformly bounded in $L^p$ for all small $p$. Property {\rm (ii)} holds for such limit solutions.
\end{theorem}

\normalcolor \begin{remark} \label{remarknonh} \textnormal{(On some nonhomogeneous boundary value problems)}. \rm Let $\varepsilon>0$. For our arguments it will
be essential to consider problems with nonhomogeneous boundary values  of the type
\begin{equation} \label{nonhomognon}
\left\{
\begin{array}
[c]{lll}%
\left(  -\Delta\right)^{\sigma/2}v+  B(v)=f\left(  x\right)   &  & in\text{ }%
\Omega,\\[6pt]
v=\varepsilon &  & in\text{ }\R^{N}\setminus\Omega\,.
\end{array}
\right. %
\end{equation}
In order to ensure the existence of a solution to this problem, we will associate to it the following nonhomogeneous extension problem
\begin{equation}
\left\{
\begin{array}
[c]{lll}%
-\operatorname{div}_{x,y}\left(  y^{1-\sigma}\nabla w\right)  =0 &  & in\text{
}{\mathcal{H^+}},\\[6pt]
\ w(x,0)=\varepsilon &  & \mbox{ for } x\in\R^{N}\setminus \Omega,\\[6pt]
\displaystyle{-\frac{1}{\kappa_{\sigma}}\lim_{y\rightarrow0^{+}}y^{1-\sigma}\,\dfrac{\partial w}{\partial y}(x,y)}+\,B(w(x,0))=f\left(  x\right)   &
&
\mbox{ for } x\in \Omega.
\end{array}
\right.  \label{nonhomogext}%
\end{equation}
If $f\in L^{1}(\Omega)$, we say that $w\in X^{\sigma/2}(\mathcal{H}^{+})$ is a weak solution to \eqref{nonhomogext} if $w-\varepsilon\in
X_{\Omega}^{\sigma/2}(\mathcal{H}^{+})$ and $w$ satisfies \eqref{weakfor}. In such case, we say that $v=\text{Tr}_{\Omega}w$ is a weak solution to problem
\eqref{nonhomognon}. In particular, if $B(t)=ct$ for all $t\geq0$ and some $c>0$, and $\overline{v}$ is the solution to the  linear problem with homogeneous
boundary data
\begin{equation*}
\left\{
\begin{array}
[c]{lll}%
\left(  -\Delta\right)^{\sigma/2}\overline{v}+  c\overline{v}=f(x)-c\varepsilon   &  & in\text{ }%
\Omega,\\[6pt]
\overline{v}=0 &  & in\text{ }\R^{N}\setminus\Omega,
\end{array}
\right. %
\end{equation*}
then $v=\overline{v}+\varepsilon$ is the unique solution to the linear, nonhomogeneous problem \eqref{nonhomognon}.\smallskip\\
Using the variational formulation to \eqref{nonhomogext}, it is easy to prove that if $\varepsilon>0$, $v_{\varepsilon}$ is the unique weak solution to
\eqref{nonhomognon} and $w_{\varepsilon}$ is its extension solving \eqref{nonhomogext}, then
\[
w_{\varepsilon}\rightarrow w\,\quad\,in\,L^{1}(\mathcal{H^{+}}),\]
\[
\text{Tr}_{\R^{N}}w_{\varepsilon}\rightarrow\text{Tr}_{\R^{N}}w,\quad\,in\,L^{1}(\R^{N}),
\]
where $v$, $w$ solve \eqref{eq.1}-\eqref{eq.3} respectively.
\end{remark}
We warn the reader that the solutions of all the Dirichlet problems throughout the paper will be identified with their extension on the whole $\R^{N}$, whose
values out of $\Omega$ will clearly depend on the boundary conditions considered.

%%%%%%%%%%%%%%%%%%%%%%%%%%%%%%
\section{Concentration comparison for the extended problem}\label{sec.ell2}

Let us address the comparison issue. From now on, we will always assume that the right-hand side $f$ is nonnegative.  Our goal here is to compare the solution
$v$ to \eqref{eq.1} with the solution $V$ to the problem
\begin{equation}
\left\{
\begin{array}
[c]{lll}%
\left(  -\Delta\right)  ^{\sigma/2}V+\, B(V)=f^{\#}\left(  x\right)   &  & in\text{ }%
\Omega^{\#}\\[6pt]
V=0 &  & on\text{ }\R^{N}\setminus\Omega^{\#}.
\end{array}
\right.  \label{eq.4}%
\end{equation}

A reasonable way to do that is to compare the solution $w$ to \eqref{eq.3} with the solution $\psi$ to the problem
\begin{equation}
\left\{
\begin{array}
[c]{lll}%
-\operatorname{div}_{x,y}\left(  y^{1-\sigma}\nabla \psi\right)  =0  &  & in\text{
}\mathcal{H}^{+}\\[6pt]
\psi(x,0)=0 &  & for\text{ }x\in\R^{N}\setminus \Omega^{\#}\\[6pt]
\displaystyle{-\frac{1}{\kappa_{\sigma}}\lim_{y\rightarrow0^{+}}y^{1-\sigma}\,\dfrac{\partial \psi}{\partial y}(x,y)}+\,B(\psi(x,0))=f^{\#}\left(
x\right)
&  & in\text{ }\Omega^{\#},
\end{array}
\right.  \label{eq.5.0}%
\end{equation}
where $\psi(x,0)=V(x)$.
According to \cite{dBVol}, using the change of variables $z=({y}/{\sigma})  ^{\sigma},$
problems \eqref{eq.3} and \eqref{eq.5.0} become respectively

\begin{equation}
\left\{
\begin{array}
[c]{lll}%
-z^{\nu}\dfrac{\partial^{2}w}{\partial z^{2}}-\Delta_{x}w=0 &  & in\text{
}\mathcal{H}^{+}\\
&  & \\
w(x,0)=0 &  & for\text{ }x\in\R^{N}\setminus \Omega^{\#}\\
&  & \\
-\dfrac{\partial w}{\partial z}\left(  x,0\right)  =%
\,\sigma^{\sigma-1}\kappa_{\sigma}\left(f\left(  x\right)-B(w(x,0))\right)  &  & in\text{ }\Omega,
\end{array}
\right.  \label{eq.23}%
\end{equation}
and
\begin{equation}
\left\{
\begin{array}
[c]{lll}%
-z^{\nu}\dfrac{\partial^{2}\psi}{\partial z^{2}}-\Delta_{x}\psi=0 &  & in\text{
}\mathcal{H}^{+}\\
&  & \\
\psi=0 &  & for\text{ }x\in\R^{N}\setminus \Omega^{\#}\\
&  & \\
-\dfrac{\partial \psi}{\partial z}\left(  x,0\right)  =%
\,\sigma^{\sigma-1}\kappa_{\sigma}\left(f^{\#}\left(  x\right)-B(\psi(x,0))\right)  &  & in\text{ }\Omega^{\#}.
\end{array}
\right.  \label{eq.24}%
\end{equation}

where $\nu:=2\left(  \sigma-1\right)  /\sigma.$ \
Then, the problem reduces to prove the concentration comparison between the solutions $w(x,z)$ and $\psi(x,z)$ to \eqref{eq.23}-\eqref{eq.24} respectively.  We
now introduce the function
\begin{equation}
{Z}(s,z)=\int_{0}^{s}(w^{\ast}(\tau,z)-\psi^{\ast}(\tau,z))d\tau\,.
\end{equation}
Using standard symmetrization tools (see \cite{dBVol}) we get the differential inequality
\begin{equation}
-(z^{\nu}{Z}_{zz}+p(  s) {Z}_{ss})\leq0\label{symineq}
\end{equation}
for a.e $(s,z)\in \left(  0,+\infty \right)
\times\left(  0,+\infty\right)  $\normalcolor,  where $p (s) = N^2\omega_N^{2/N} s^{2- 2/ N}$. Obviously, we have
\begin{equation}
Z(0,y)=0\label{boundcond}.
\end{equation}
A crucial point in our arguments below is played by the derivative of $Z$ with respect to $z$. Due to the boundary conditions contained
in
\eqref{eq.23}-\eqref{eq.24}, we have for $0<s<|\Omega|$
\begin{equation}\label{Z_yboundary.formula}
{Z}_{z}(s,0)\geq \theta_{\sigma}\int_{0}^{s} (B(w^*(\tau,0))-B(\psi^{\ast}(\tau,0))\, d\tau
\end{equation}
where $\theta_{\sigma}:=\sigma^{\sigma-1}\kappa_{\sigma}$; on the other hand,
for $s\ge |\Omega|$ we have $(\partial{Z}/\partial s\,)(s,0)=0$. This is the novelty in the argument with respect to the spectral case treated before, where the
boundary condition
\[
\frac{\partial Z}{\partial{s}}(|\Omega|,z)=0,\quad\forall z\geq0
\]
made the real difference: the boundary conditions are imposed on $y=0$ and are of two types. Contradiction must be obtained after taking into account the two
possibilities.

\medskip

%%%%%%%%%%%%%%%%%%%%%%%%%%%%%%%%%%%%%%%%%%%%%%%%%%%%%%%%%%%%%%%%%%%%

\subsection{Comparison result for the elliptic problem}

We are going to obtain a comparison result for some linear and nonlinear $B$. Actually the nonlinearities considered here will allow to get a result which is
weaker than the one for the {\sl linear} problem, i.\,e., when $B(t)=ct$ for some $c\geq0$, which is the only case that we are going to need in  addressing the
Faber-Krahn inequalities. We also point out that, in order to reach our goal, we will use a lifting-type argument of the symmetrized problem \eqref{eq.4}

\begin{theorem}\label{thm.ell.linear} Let $v$ be the nonnegative solution of  problem $\eqref{eq.1}$ posed in a bounded domain with zero
Dirichlet
boundary
condition, nonnegative data $f\in L^1(\Omega)$ and $B$ is smooth, concave, strictly increasing on $\R_{+}$ and such that $B(0)=0$. If $V$
is the solution of the corresponding symmetrized problem \eqref{eq.4}, we have
\begin{equation}
v^\#(x)\prec V(x).\label{concinequlin}
\end{equation}
The same is true if $\Omega=\ren$.
\end{theorem}

\begin{proof} In this case we pose the problem first in a bounded domain $\Omega$ of $\ren$ with smooth boundary. We also assume that $f$ is smooth, bounded and
compactly supported in $\Omega$, since the comparison result for general data can be obtained later by approximation using the $L^1$ dependence of the map
$f\mapsto v$. Let us choose $\varepsilon>0$,
$f_{\varepsilon}=f+B(\varepsilon)$ and let us consider the solution $V_{\varepsilon}$ to the problem
\begin{equation*}
\left\{
\begin{array}
[c]{lll}%
\left(  -\Delta\right)  ^{\sigma/2}V_{\varepsilon}+B(V_{\varepsilon})=f_{\varepsilon}^{\#}\left(  x\right)   &  & in\text{ }%
\Omega^{\#}\\[6pt]
V_{\varepsilon}=\varepsilon &  & on\text{ }\R^{N}\setminus\Omega^{\#},
\end{array}
\right.  \label{eq.4}%
\end{equation*}
where $f_{\varepsilon}^{\#}=f^{\#}+B(\varepsilon)$. By virtue of Remark \ref{remarknonh} we have  $V_{\varepsilon}=\psi_{\varepsilon}(x,0)$
for all $x\in \R^{N}$, where $\psi_{\varepsilon}$ is the solution to the problem
\begin{equation*}
\left\{
\begin{array}
[c]{lll}%
-\operatorname{div}_{x,y}\left(  y^{1-\sigma}\nabla \psi_{\varepsilon}\right)  =0  &  & in\text{
}\mathcal{H}^{+}\\[6pt]
\psi_{\varepsilon}(x,0)=\varepsilon
 &  & for\text{ }x\in\R^{N}\setminus \Omega^{\#}\\[6pt]
\displaystyle{-\frac{1}{\kappa_{\sigma}}\lim_{y\rightarrow0^{+}}y^{1-\sigma}\,\dfrac{\partial \psi_{\varepsilon}}{\partial
y}(x,y)}+B(\psi_{\varepsilon}(x,0))=f^{\#}\left(
x\right)+B(\varepsilon)
&  & in\text{ }\Omega^{\#},
\end{array}
\right.  \label{eq.5}%
\end{equation*}
which can be reduced to the problem
\begin{equation}
\left\{
\begin{array}
[c]{lll}%
-z^{\nu}\dfrac{\partial^{2}\psi_{\varepsilon}}{\partial z^{2}}-\Delta_{x}\psi_{\varepsilon}=0 &  & in\text{
}\mathcal{H}^{+}\\
&  & \\
\psi_{\varepsilon}(x,0)=\varepsilon &  & for\text{ }x\in\R^{N}\setminus \Omega^{\#}\\
&  & \\
-\dfrac{\partial \psi_{\varepsilon}}{\partial z}\left(  x,0\right)  =%
\,\sigma^{\sigma-1}\kappa_{\sigma}\left(f^{\#}\left(  x\right)+B(\varepsilon)-B(\psi_{\varepsilon}(x,0))\right)  &  & in\text{ }\Omega^{\#}.
\end{array}
\right.  \label{pbepsnonlin}%
\end{equation}
Setting
\[
{Z}_{\varepsilon}(s,z)=\int_{0}^{s}(w^{\ast}(\tau,z)-\psi_{\varepsilon}^{\ast}(\tau,z))d\tau\,
\]
We will prove that
\begin{equation}
Z_{\varepsilon}(s,z)\leq0\label{ineqZeps}.
\end{equation}
To this aim, we first observe that $Z_{\varepsilon}$ satisfies inequality \eqref{symineq}. In particular, a big role is played by the property of the solution
$\psi_{\varepsilon}$: in particular, the level sets $\left\{x\in\R^{N}\,:\psi_{\varepsilon}(x,z)>t\right\}$ are bounded because they are balls centered at the
origin.
Now we set
\[
Y_{\varepsilon}(s,z)=\int_{0}^{s}(B(w^*(\tau,z))-B(\psi_{\varepsilon}^*(\tau,z)))\, d\tau.\]
From \eqref{Z_yboundary.formula}, we obtain
\begin{equation}\label{Zeps_yboundary.formula}
\frac{\partial Z_{\varepsilon}}{\partial z}(s,0)\geq \theta_{\sigma}Y_{\varepsilon}(s,0)+\theta_{\sigma}B(\varepsilon).
\end{equation}
By the strong maximum principle applied to equation \eqref{symineq}, , which is satisfied by $Z_{\varepsilon}$, a positive maximum of $Z_{\varepsilon}$ cannot
be
achieved at an interior point, hence it must be achieved either as $s\rightarrow\infty$ or at a
boundary point $(s_0, 0)$ for some $s_0 > 0$. The first option cannot hold since $w^{\ast}(s,z)\rightarrow0$ while $\psi_{\varepsilon}(s,z)\rightarrow
\varepsilon$ as $s\rightarrow\infty$.
As for the second, we see that for $s\geq|\Omega|$ we have
\[
\frac{\partial Z_{\varepsilon}}{\partial s}(s,0)=v^{\ast}(s)-V_{\varepsilon}^{\ast}(s)=-\varepsilon<0,
\]
the function $Z_{\varepsilon}(s,0)$ is strictly decreasing in $[|\Omega|,\infty)$, then it must happen that  $s_0\in (0,|\Omega|]$. Arguing as  as in
\cite{VazVol1}, using \eqref{Zeps_yboundary.formula}, we
can
write
\begin{align}
&\frac{\partial{Z_{\varepsilon}}}{\partial z}(s_{0},0)>\theta_{\sigma}\int_{0}^{s_0}B^{\prime}(v^{\ast}(s))\frac{\partial{Z_{\varepsilon}}}{\partial
s}(s,0)ds\nonumber\\
&=\theta_{\sigma}\left[g(0)Z_{\varepsilon}(s_{0},0)+\int_{0}^{s_0}[Z_{\varepsilon}(s_{0},0)-Z_{\varepsilon}(s,0)]dg(s)\right]>0\label{relmaxZY}
\end{align}
where $g(s):=B^{\prime}(v^{\ast}(s))$, which is impossible due to Hopf's boundary maximum principle.\\
Finally, by \eqref{ineqZeps}
we have, for $s\geq 0$
\[
\int_{0}^{s}v^{\ast}(\tau)d\tau\leq \int_{0}^{s}V_{\varepsilon}^{\ast}(\tau)d\tau
\]
thus as $\varepsilon\rightarrow0$
\[
\int_{0}^{s}v^{\ast}(\tau)d\tau\leq\int_{0}^{s}V^{\ast}(\tau)d\tau.
\]
and the result follows. The case $\Omega=\R^{N}$ has been solved in \cite[Theorem 3.2]{VazVol1}, to which we address the interested reader.
\end{proof}
\begin{remark} If we consider the linear case, \emph{i.e.} when $B(t)=ct$, for some $c\geq0$, the proof of Theorem \ref{thm.ell.linear} can be simplified,
because $Z_{\varepsilon}=cY_{\varepsilon}$, moreover $\psi_{\varepsilon}=\psi+\varepsilon\label{linearidn}$ where $\psi$ solves \eqref{eq.24}, and
$V_{\varepsilon}=V+\varepsilon$.
\end{remark}
\begin{remark}
If $B$ is nonlinear and satisfies the assumptions of Theorem \ref{thm.ell.linear}, the mass concentration comparison \eqref{concinequlin} is not enough to obtain
the same result when passing to the evolution problem via Crandall-Liggett Theorem.
\end{remark}

\normalcolor

%%%%%%%%%%%%%%%%%%%%%%%%%%%%%% added march 2015 %%%%%%%%%%
A valuable property we are going to prove now is that equality in \eqref{concinequlin}  implies that the domain $\Omega$ of the initial problem \eqref{eq.1} is
actually a {\sl ball} modulo translation. This kind of property is known to be essential for studying the equality case in the Faber--Krahn
inequality, {i.\,e.,} to prove that the ball is the unique minimizer of the principal eigenvalue of the classical Laplacian over all the sets of  fixed Lebesgue
measure (see for instance Kesavan \cite{Kes}). %%%%%%%%%%%%%%%%%%%%%%%%%%

\begin{proposition}[Case of equality]\label{equality}
Assume that we have an equality sign in \eqref{concinequlin}, in the sense that
\[
\int_{0}^{s}v^{*}(\sigma)d\sigma=\int_{0}^{s}V^{*}(\sigma)d\sigma
\]
for all $s\in [0,|\Omega|]$. Then $\Omega$ is a ball, i.e. $\Omega=\Omega^{\#}$ (up to a translation of the origin).
\end{proposition}

\begin{proof}
Using the same notation as in Theorem \ref{thm.ell.linear}, by the given assumption we have $Y(s,0)=0$ for all $s\in [0,|\Omega|]$, and since the extensions are
null, this equality holds for every $s\geq0$. Then by Hopf's maximum principle and \eqref{Z_yboundary.formula} we find $Y\equiv0$, thus
\begin{equation}
w^{\ast}(s,z)=\psi^{*}(s,z)\quad \mbox{for all } \ s,z\geq0.\label{eqrearr}
\end{equation}

We divide the rest into some steps:  (i) Here we argue as in \cite{Talenti1} or in \cite{Mossino}. Recall that the function $w(\cdot,z)$ is smooth on $\R^{N}$
for any $z>0$. Then let us fix $z>0$,   multiply both sides of the equation \eqref{eq.23} by the test function
\[
\varphi_{h}^{z}\left(  x\right)  =\left\{
\begin{array}
[c]{lll}%
1 &  & if\text{ \ }w(x,z) \geq t+h\\
&  & \\
\dfrac{ w(x,z) -t}{h}\, &  & if\text{
\ }t< w(x,z) <t+h\\
&  & \\
0 &  & if\text{ \ } w(x,z) \leq t,\text{ }%
\end{array}
\right.
\]
and integrate over $\R^{N}$. An integration by parts yields
the identity
$$\begin{array}{c}
\frac{1}{h}\dint_{\left\{x\in \R^{N}:\, t< w(x,z) <t+h\right\}}\left\vert \nabla
_{x}w\right\vert ^{2}dx-z^{\nu}\dint_{\left\{x\in\R^{N}:\, w(x,z)
>t+h\right\}}\dfrac{\partial^{2}w}{\partial z^{2}}dx\\[12pt]
-z^{\nu}\dint_{\left\{x\in\R^{N}:\,t< w(x,z) <t+h\right\}}\dfrac{\partial^{2}w}{\partial z^{2}%
}\left(  \dfrac{w -t}{h}\right)  dx=0.
\end{array}
$$
Then if we let $h\rightarrow0$ we find
\[
-\frac{\partial}{\partial t}\int_{\left\{x\in\R^{N}:\,w(x,z)>t\right\}}|\nabla_{x}w|^2\,dx=z^{\nu}\int_{\left\{x\in\R^{N}:\, w(x,z)
>t\right\}}\dfrac{\partial^{2}w}{\partial z^{2}}dx
\]
thus using the second order derivation formula (see Appendix) we get
\begin{equation*}
-\frac{\partial}{\partial t}\int_{\left\{x\in\R^{N}:\,w(x,z)>t\right\}}|\nabla_{x}w|^2\,dx\leq
z^{\nu}\int_{0}^{\mu_{w}(t,z)}\frac{\partial^{2}w^{\ast}}{\partial
z^{2}}ds.
\end{equation*}
Concerning the solution $\psi$ to problem \eqref{eq.24}, since it is spherically decreasing w.r.to $x$, the follow equality occurs
\begin{equation*}
-\frac{\partial}{\partial t}\int_{\left\{x\in\R^{N}:\,\psi(x,z)>t\right\}}|\nabla_{x}\psi|^2\,dx=
z^{\nu}\int_{0}^{\mu_{\psi}(t,z)}\frac{\partial^{2}\psi^{\ast}}{\partial z^{2}}ds.
\end{equation*}
then using the fact that $w^{\ast}(s,z)=\psi^{*}(s,z)$ (which implies $\mu_{w}(\cdot,z)=\mu_{\psi}(\cdot,z)$) we have
\begin{equation*}
-\frac{\partial}{\partial t}\int_{\left\{x\in\R^{N}\,:w(x,z)>t\right\}}|\nabla_{x}w|^2\,dx\leq-\frac{\partial}{\partial
t}\int_{\left\{x\in\R^{N}:\,\psi(x,z)>t\right\}}|\nabla_{x}\psi|^2\,dx
\end{equation*}
Integrating between $t$ and $\infty$, we find
\[
\int_{\left\{x\in\R^{N}\,:w(x,z)>t\right\}}|\nabla_{x}w|^2\,dx\leq \int_{\left\{x\in\R^{N}:\,\psi(x,z)>t\right\}}|\nabla_{x}\psi|^2\,dx
\]
then by the P\'olya-Szeg\"o inequality and \eqref{eqrearr}
\begin{align}
&\int_{\left\{x\in\R^{N}\,:w^{\#}(x,z)>t\right\}}|\nabla_{x}w^{\#}|^2\,dx\leq\int_{\left\{x\in\R^{N}\,:w(x,z)>t\right\}}|\nabla_{x}w|^2\,dx\leq
\int_{\left\{x\in\R^{N}:\,\psi(x,z)>t\right\}}
|\nabla_{x}\psi|^2\,dx\nonumber\\&=
\int_{\left\{x\in\R^{N}\,:w^{\#}(x,z)>t\right\}}|\nabla_{x}w^{\#}|^2\,dx\,.
\end{align}
We conclude that for every $t>0$
\begin{equation*}
\int_{\left\{x\in\R^{N}\,:w(x,z)>t\right\}}|\nabla_{x}w|^2\,dx=
\int_{\left\{x\in\R^{N}\,:w^{\#}(x,z)>t\right\}}|\nabla_{x}w^{\#}|^2\,dx\,,
\end{equation*}
which is the \emph{equality} case in the P\'olya-Szeg\"o inequality.

(ii) Now we notice that $w(\cdot,z)$ is analytic on the upper half space as a consequence of its representation formula. Indeed, the  Poisson kernel for the
extension operator $L_\sigma$ is given by
$$
P(x, z) = c_{\sigma,N}\frac{z^\sigma}{(|x|^2 + z^2)^{N+\sigma}}
$$
(see e.\,g. \cite{CaffSilv}), which is an analytic function on the upper half-space $z>0$. This means that for every $h(x)\in L^1(\re^{N})$ the solution
$w(x,z)$
of the elliptic equation with data $w(x,0)=h(x)$ will be analytic, since it is the convolution of $h$ with $P$ (with respect to the $x$ variable).
In fact, once we know that $w\in C^\infty$ in a certain subdomain, it will be analytic by classical results on solutions of elliptic equations with analytic
coefficients (see for instance \cite{Petr}, \cite{Fried}, \cite{Hashimoto}).

(iii) Moreover, each level set $\left\{x\in\R^{N}\,:w(x,z)>t\right\}$ is bounded because $w(\cdot,z)$ decays to zero as $|x|\to \infty$.  We may now use the
equality case in the P\'olya-Szeg\"{o} inequality (see \cite{BroZie}, \cite{FeroneVolp}) to obtain that
$\left\{x\in\R^{N}\,:w(x,z)>t\right\}=\left\{x\in\R^{N}\,:w^{\#}(x,z)>t\right\}$ modulo a translation. Then all the level sets
${\left\{x\in\R^{N}\,:w(x,z)>t\right\}}$ are balls. The results also imply that
 for every fixed $z>0$ the function $w(\cdot,z)$ is radially symmetric up to translation.

(iv) Finally, we take the limit $z\to 0$ and we conclude that $u(x)$ is also radially symmetric (as a function defined in $\ren$). This means that the domain
$\Omega$, which is the positivity set of $u$ in $\ren$, must be a ball, i.\,e., $\Omega=\Omega^{\#}$, up to a translation of the origin.
\end{proof}

The following result can be shown as in the proof of \cite[Theorem 3.3]{VazVol1}

\begin{theorem}[Comparison of concentrations for radial problems]\label{thm.ell.linear.rad} Let $v_1, v_2$ be two nonnegative solutions of  problem
$\eqref{eq.1}$ posed in a ball $B_{R}(0)$,
with
$R\in(0,+\infty]$ with zero Dirichlet boundary conditions if $R<+\infty$, nonnegative radially symmetric decreasing data $f_1, f_2\in L^1(B_{R}(0))$
and $B(t)=ct$ for some $c>0$ and all $t\geq0$. Then $v_1$ and $v_2$ are rearranged, and
\begin{equation}
 f_1\prec f_2 \quad \mbox{implies} \quad v_1\prec v_2\,.
\end{equation}
\end{theorem}

%%%%%%%%%%%%%%%%%%%%%%%%%%%%%%%%%%%%%%%%%%%%%%%%%%%%%%%%%%%%

%%%%%%%%%%%%%%%%%%%%%%%%%%%%%%%%%%%%%%%%%%%%%%%%%%%%%%%%%%%%%%%%%%%
%%%%%%%%%%%%%%%%%%%%%%%%%%%%%%%%%%%%%%%%%%%%%%%%%%%%%%%%%%%%%%%%%%%

\section{Symmetrization for the parabolic problem}\label{sec.par}
\setcounter{equation}{0}

 The theory of existence of weak solutions for the initial value problem
 \begin{equation}\label{nolin.parab2}
\partial_t u +(-\Delta)^{\sigma/2}A(u)=f, \qquad 0<\sigma<2\,.
\end{equation}
with $A(u)=cu^m$, and all $c,m>0$,  has been addressed by the first author and collaborators in \cite{BonfSireVaz2014}, and the main properties have been
obtained. In particular, if we take initial data in $L^1\cap H^{-s}$, then an $H^{-s}$-contraction semigroup is generated, and the Crandall-Liggett
discretization theorem applies. The construction and properties of the solutions of the evolution problem is thus reduced to an iterated application of the
results obtained for the elliptic counterpart in the previous section. This has been carefully explained in \cite{VazVol1} and is reviewed in Appendix III. Thus,
using Theorem \ref{existmildsol} below, we obtain the existence of a unique mild
solution to the linear Cauchy-Dirichlet problem on the bounded domain $\Omega$:
\begin{equation} \label{eqcauchyDirich}
\left\{
\begin{array}
[c]{lll}%
u_t+(-\Delta)^{\sigma/2}u=f  &  & x\in\Omega\,,t>0%
\\[6pt]
u=0 &  & x\in\R^{N}\setminus\Omega\,,t>0\\[6pt]
u(x,0)=u_{0}(x) &  & x\in\Omega.
\end{array}
\right. %
\end{equation}
obtained as a limit of discrete approximate solutions by the ITD scheme.

Concerning the application of symmetrization techniques to this type of parabolic problems, we can employ Theorems \ref{thm.ell.linear},
\ref{thm.ell.linear.rad}
and the arguments in the proof of \cite[Theorem 5.3]{VazVol1} in order to find the following result

\begin{theorem}[Concentration comparison]\label{conccomplin} Let $u$ be the nonnegative mild solution to problem \eqref{eqcauchyDirich}, $0<\sigma<2$, posed in
$\Omega$, with
initial data $u_0\in L^1(\Omega)$, right-hand side $f\in L^1(\Omega\times(0,\infty))$. Assume that $\overline{u}_{0}\in L^{1}(\Omega^{\#})$ is rearranged,
$\overline{f}(x,t)\in L^{1}(\Omega^{\#}\times(0,\infty)$ is rearranged w.r. to any $t>0$, such that
\[
u_{0}^{\#}\prec \overline{u}_{0}
\]
and
\[
f^{\#}(\cdot,t)\prec \overline{f}(\cdot,t)
\]
for a.e. $t>0$. Let $v$ be the solution of the evolution problem
\begin{equation} \label{eqcauchysymm.f}
\left\{
\begin{array}
[c]{lll}%
v_t+(-\Delta)^{\sigma/2}v=\overline{f}(x,t)  &  & x\in\Omega^{\#}\,, \ t>0,%
\\[6pt]
v=0 &  & x\in\R^{N}\setminus\Omega^{\#}\,,t>0
\\[6pt]
v(x,0)=\overline{u}_{0}(x) &  & x\in\Omega^{\#},
\end{array}
\right. %
\end{equation}
Then, for all $t>0$ we have
\begin{equation}
u^\#(|x|,t)\prec v(|x|,t).
\end{equation}
In particular, we have $\|u(\cdot,t)\|_p \le\|v(\cdot,t)\|_p$ for every $t>0$ and every $p\in [1,\infty]$.
\end{theorem}

\noindent {\bf Remark.} The parabolic result only covers the linear equation, which is much below our
original expectations, since the elliptic result covers indeed nonlinear equations. We do not know if the lack of nonlinear results is due to a failure of the
expected form of the  theorem, as copied from \eqref{conccomplin}, or is only due to a lack of technique. We remind the reader that in the case of the problem in
the whole space we were able to prove the nonlinear symmetrization result when $A$ is concave, and to show that the corresponding statement for $A$ convex is
false, so that the question
was posed about which kind of statement could be true. Such question is broader in the present case.

%%%%%%%%%%%%%%%%%%%%%%%%%%%%%%%%%%%%%%%%%%%%%%%%%%%%%%%%%%%%%%%%%%%%%%%

\section{Application: an original proof of the fractional Faber-Krahn inequality}\label{sect.fk}

Here we prove the validity of the Faber-Krahn inequality for the fractional Laplacian operator ${\cal L}_2$ defined on bounded domains of $\ren$ with zero
Dirichlet boundary
conditions. This operator  appears often in theory and applications, and is known under the name of {\sl restricted fractional Laplacian}, though we can call it
the {\sl natural
fractional Laplacian}. Unlike the spectral Laplacian  ${\cal L}_1$, the  spectral sequence $\{\lambda_k({\cal L}_2; \Omega)\}_k$ is not directly  related to the
sequence of the standard Laplacian.  However, it is known that the spectrum is discrete
and given by a strictly increasing sequence $\left\{\lambda_{k,\sigma/2}(\Omega)=\lambda_k(\mathcal{L}_{2}; \Omega)\right\}$ (see e.\,g.
\cite{BonfSireVaz2014}).
Our Theorem is then stated as follows:

\begin{theorem}\label{FaberKrahn}
We have
\begin{equation}
\lambda_{1,\sigma/2}(\Omega)\geq\lambda_{1,\sigma/2}(\Omega^{\#})\label{Faber}
\end{equation}
with equality if and only if $\Omega=\Omega^{\#}$, up to translation.
\end{theorem}

The proof we present here is completely elementary and uses neither the variational characterization \eqref{firsteigenv} nor the nonlocal P\'{o}lya-Sz\"{e}go
inequality as in \cite{BrascoParini}, but only the concentration comparison provided by Theorem \ref{conccomplin} and the asymptotic definition of
$\lambda_{1,\sigma/2}(\Omega)$, which
can be derived by the decay rate of the parabolic problem $u_t+{\cal L}_{2} u=0$ .
\normalcolor

\noindent {\sl Proof of Theorem \ref{FaberKrahn}.\,\,}
Suppose that $\left\{\psi_{k,\sigma/2,\Omega}(x)\right\}_k$ are the ($L^{2}$ normalized) eigenfunctions of $\mathcal{L}_{2}$ and let us consider the function
\[
u(x,t)=e^{-\lambda_{1,\sigma/2}(\Omega)\,t}\,\psi_{1,\sigma/2,\Omega}(x)
\]
solving the problem
\begin{equation} \label{eqcauchyDirich}
\left\{
\begin{array}
[c]{lll}%
u_t+(-\Delta)^{\sigma/2}u=0  &  & x\in\Omega\,,t>0%
\\[6pt]
u=0 &  & x\in\R^{N}\setminus\Omega\,,t>0\\[6pt]
u(x,0)=\psi_{1,\sigma/2,\Omega}(x) &  & x\in\Omega.
\end{array}
\right. %
\end{equation}
By Theorem \ref{conccomplin} we find, for all $t>0$,
\[
\|u(\cdot,t)\|_{2}\leq \|v(\cdot,t)\|_{2},
\]
where $v$ solves the problem
\begin{equation} \label{eqcauchysymm.f}
\left\{
\begin{array}
[c]{lll}%
v_t+(-\Delta)^{\sigma/2}v=0  &  & x\in\Omega^{\#}\,, \ t>0,%
\\[6pt]
v=0 &  & x\in\R^{N}\setminus\Omega^{\#}\,,t>0
\\[6pt]
v(x,0)=\psi^{\#}_{1,\sigma/2,\Omega}(x) &  & x\in\Omega^{\#}.
\end{array}
\right. %
\end{equation}
Since we have
\begin{equation}
\|u(\cdot,t)\|_{2}=e^{-\lambda_{1,\sigma/2}(\Omega)\,t}\label{comparisonFaber}
\end{equation}
and the expression for $v$ can be given in terms of superposition, namely
\[
v(x,t)=\sum_{k=1}^{\infty}\langle \psi^{\#}_{1,\sigma/2,\Omega},\psi_{k,\sigma/2,\Omega^{\#}}\rangle_{L^{2}(\Omega^{\#})}\,
e^{-\lambda_{k,\sigma/2}(\Omega^{\#})\,t}\,\psi_{k,\sigma/2,\Omega^{\#}}(x)
\]
thus
\begin{align*}
&\|v(\cdot,t)\|^{2}_{2}=\sum_{k=1}^{\infty}|\langle \psi^{\#}_{1,\sigma/2,\Omega},\psi_{k,\sigma/2,\Omega^{\#}}\rangle|^{2}_{L^{2}(\Omega^{\#})}\,
e^{-2\lambda_{k,\sigma/2}(\Omega^{\#})\,t}\\
&\leq e^{-2\lambda_{1,\sigma/2}(\Omega^{\#})\,t}\sum_{k=1}^{\infty} |\langle
\psi^{\#}_{1,\sigma/2,\Omega},\psi_{k,\sigma/2,\Omega^{\#}}\rangle|^{2}_{L^{2}(\Omega^{\#})}\\
&=e^{-2\lambda_{1,\sigma/2}(\Omega^{\#})\,t}\|\psi^{\#}_{1,\sigma/2,\Omega}\|^{2}_{L^{2}(\Omega^{\#})}
=e^{-2\lambda_{1,\sigma/2}(\Omega^{\#})\,t}\,.
\end{align*}
This together with \eqref{comparisonFaber} implies \eqref{Faber}.

\medskip

\noindent {\sl The case of equality.}
Analyzing the last list of inequalities that starts by  $\|v(\cdot,t)\|^{2}_{2}$
we conclude that in the case where $\lambda_{1,\sigma/2}(\Omega)=\lambda_{1,\sigma/2}(\Omega^{\#})$ we necessarily have
\[
\|v(\cdot,t)\|_{L^{2}(\Omega^{\#})}=\|u(\cdot,t)\|_{L^{2}(\Omega)}=e^{-\lambda_{1,\sigma/2}(\Omega^{\#})\,t}
\]
so that we conclude that the coefficients of the Fourier expansion of $v_0=v(\cdot,0)$ in terms of eigenfunctions are all zero but the first, in view of the
known  fact that the first
eigenvalue $\lambda_{1,\sigma/2}(\Omega^{\#})$ is simple. This means that $v_0(x)=c \psi_{1,\sigma/2,\Omega^{\#}}$, with a constant $c>0$. By normalization we
get
\begin{equation}
\psi^{\#}_{1,\sigma/2,\Omega}(x)=\psi_{1,\sigma/2,\Omega^{\#}}(x)\,.
\end{equation}
 But this is enough to apply the important Proposition \ref{equality} and obtain $\Omega=\Omega^{\#}$, and the result ends as before. \qed

\medskip

\subsection{A variational proof} A direct proof of the FKI for our  operator and similar can be  based on the variational interpretation of the first eigenvalue,
since it  can be written as the minimizer of the Rayleigh quotient, where the local $L^{2}$ gradient energy norm is
replaced by the Gagliardo seminorm (see Section \ref{sec.prelim} for some details).

Suppose that $\left\{\psi_{k,\sigma/2,\Omega}(x)\right\}_k$ are the ($L^{2}$ normalized) eigenfunctions of $\mathcal{L}_{2}$.
As already mentioned in the introduction, the proof of Theorem \ref{FaberKrahn} is a direct consequence of the variational characterization of
$\lambda_{1,\sigma/2}(\Omega)$. Indeed, we know by \cite{SerValvariat} that
\begin{equation}
\lambda_{1,\sigma/2}(\Omega)=\min_{\substack{u\in H^{\sigma/2}(\ren)\setminus\left\{0\right\}\\ u=0\text{ on
}\ren\setminus\Omega}}\frac{\displaystyle{\int_{\R^{N}}\int_{\R^{N}}\dfrac{|u(x)-u(y)|^{2}}{|x-y|^{N+\sigma}}\,dxdy}}{\displaystyle
{\int_{\Omega}u^{2}dx}}.\label{firsteigenv}
\end{equation}
Then we could use the nonlocal (hence fractional) version of the P\'{o}lya-Sz\"{e}go inequality (see for instance \cite{Park}) to see that replacing $u$ with
$u^{\#}$ makes the Gagliardo seminorm in \eqref{firsteigenv} decrease, therefore \eqref{Faber} holds. Furthermore if equality occurs in \eqref{Faber}, the
minimality of $\lambda_{1,\sigma/2}(\Omega^{\#})$ implies that $\psi^{\#}_{1,\sigma/2,\Omega}(x)$ is an eigenfunction, but since the eigenvalue
$\lambda_{1,\sigma/2}(\Omega^{\#})$ is also simple, by normalization we have
\[
\psi^{\#}_{1,\sigma/2,\Omega}(x)=\psi_{1,\sigma/2,\Omega^{\#}}(x)
\]
which the same result in \cite{Park} shows to be possible only when $\Omega=\Omega^{\#}$ and $\psi_{1,\sigma/2,\Omega}=\psi^{\#}_{1,\sigma/2,\Omega}$ up to
translation.

\medskip

\noindent {\bf  More general version of the FKI.} Actually,   Brasco et al. \cite{BrascoParini} are able to establish a more general version of the
FKI, which applies to a nonlinear variant of the fractional Laplacian, namely the fractional $p-$Laplacian. The main argument of the general variational proof is
the the use of a nonlocal P\'{o}lya-Sz\"{e}go inequality, proved in \cite{Frank}, to estimate the first nonlinear eigenvalue. 

\medskip

\normalcolor
\noindent {\bf Probabilistic  approach.} The Faber-Krahn inequality for the fractional Laplacian with Dirichlet data on a bounded domain of $\ren$ is stated
with
a hint of the proof based on probabilistic arguments as the last result, Theorem 5, of \cite{Banuelos}. This means that the eigenvalues are also characterized
in
terms of the evolution, in their case the stochastic process of Levy type.
Another proof with probabilistic flavor can be found in \cite{Betsakos}.\normalcolor

%%%%%%%%%%%%%%%%%%%%%%%%%%%%%%
\section{Appendices} \label{sec.prelim}
\setcounter{equation}{0}

%%%%%%%%%%%%%%%%%%%%%%%%%%%%%%%%%%%%%%%%%%%%%%%%%%%%%%%%%%%%%%%%%%%%%%%%%%%%%%%%%%%%%%%%%%%%%%%%%%%%%%%%%%%
\noindent {\bf I. On symmetrization}

\noindent We gather here some basic information on symmetrization that can be useful to read the paper. We follow standard notations used in the literature, and
we recall that we have presented a more detailed account in \cite{VazVol1}. A measurable real function $f$ defined on $\R^{N}$ is called \emph{radially
symmetric} (\emph{radial},  for short) if there is a function
$\widetilde{f}:[0,\infty)\rightarrow \R$ such that $f(x)=\widetilde{f}(|x|)$ for all $x\in \R^{N}$. We will often write $f(x)=f(r)$, $r=|x|\ge0$ for such
functions by abuse of notation. We say that $f$ is \emph{rearranged} if it is radial, nonnegative and $\widetilde{f}$ is a right-continuous,
non-increasing function of $r>0$. A similar definition can be applied for real functions defined on a ball $B_{R}(0)=\left\{x\in\R^{N}:|x|<R\right\}$.

If  $\Omega$ is an open set of $\mathbb{R}^{N}$ and $f$ is a real measurable function on $\Omega$. We will denote by $\left\vert \cdot\right\vert $ the
$N$-dimensional Lebesgue measure. We define the
\emph{distribution function} $\mu_{f}$ of $f$ as%
\[
\mu_{f}\left(  k\right)  =\left\vert \left\{  x\in\Omega: \vert f\left(
x\right)\vert>k\right\}  \right\vert \text{ , }k\geq0,
\]
and the \emph{decreasing rearrangement} of $f$ as%
\[
f^{\ast}\left(  s\right)  =\sup\left\{ k\geq0:\mu_{f}\left(  k\right)
>s\right\}  \text{ , }s\in\left(  0,\left\vert \Omega\right\vert \right).
\]
We may also  think of extending $f^{\ast}$ as the  zero function in $[|\Omega|,\infty)$ if $\Omega$ is bounded. From this definition it turns out
that $\mu_{f^{\ast}}=\mu_{f}$ (\emph{i.\,e.\,,} $f$, and $f^{\ast}$ are equi-distributed) and $f^{\ast}$ is exactly  the \emph{generalized
inverse} of
$\mu_{f}$.
Furthermore, if $\omega_{N\text{ }}$ is the measure of the unit ball in $%
%TCIMACRO{\U{211d} }%
%BeginExpansion
\mathbb{R}
%EndExpansion
^{N}$ and $\Omega^{\#}$ is the ball of $%
%TCIMACRO{\U{211d} }%
%BeginExpansion
\mathbb{R}
%EndExpansion
^{N}$ centered at the origin having the same Lebesgue measure as $\Omega,$ we define the
function
\[
f^{\#}\left(  x\right)  =f^{\ast}(\omega_{N}\left\vert x\right\vert
^{N})\text{ \ , }x\in\Omega^{\#},
\]
that will be called \emph{spherical decreasing rearrangement} of $f$. From this definition it follows that $f$ is rearranged if and only if $f=f^{\#}$.

Rearranged functions have a number of interesting properties. Here, we just recall
the conservation of the
$L^{p}$
norms (coming from the definition of rearrangements and the classical \emph{Cavalieri principle}): for all $p\in[1,\infty]$
\[
\|f\|_{L^{p}(\Omega)}=\|f^{\ast}\|_{L^{p}(|0,\Omega|)}=\|f^{\#}\|_{L^{p}(\Omega^{\#})}\,,
\]
as well as the classical Hardy-Littlewood inequality (see \cite{MR0046395})%
\begin{equation}
\int_{\Omega}\left\vert f\left(  x\right)  g\left(
x\right)  \right\vert dx\leq\int_{0}^{\left\vert \Omega\right\vert }f^{\ast
}\left(  s\right)  g^{\ast}\left(  s\right)  ds=\int_{\Omega^{\#}}f^{\#}(x)\,g^{\#}(x)\,dx\,.
\label{HardyLit}%
\end{equation}

\noindent $\bullet$ We will often deal with two-variable functions of the type
\begin{equation}\label{f}%
f:\left(  x,y\right)  \in\mathcal{C}_{\Omega}\rightarrow f\left(  x,y\right)
\in{\mathbb{R}}
\end{equation}
defined on the cylinder  $\mathcal{C}%
_{\Omega}:=\Omega\times\left(  0,+\infty\right)  $, and measurable with respect to
$x.$ Here $\Omega$ can be a bounded domain or $\ren$. For such functions, it will be convenient to define the so-called {\sl Steiner symmetrization} of
$\mathcal{C}_{\Omega}$ with
respect to the variable $x$, namely the set \ \hbox{$\mathcal{C}_{\Omega}^{\#}:=\Omega
^{\#}\times\left(  0,+\infty\right).$} Furthermore, we will denote by $\mu
_{f}\left(  k,y\right)  $ and $f^{\ast}\left(  s,y\right)  $ the distribution
function and the decreasing rearrangements of (\ref{f}), with respect to $x$
for $y$ fixed, and we will also define the function%
\[
f^{\#}\left(  x,y\right)  =f^{\ast}(\omega_{N}|x|^{N},y)
\]
which is called the \emph{Steiner symmetrization} of $f$, with respect to the line
$x=0.$ Clearly, $f^{\#}$ is a spherically symmetric and decreasing function
with respect to $x$, for any fixed $y$.

\noindent $\bullet$ There are  some interesting differentiation formulas which turn out to be very useful in our approach. Typically, they are used when one
wants to get sharp estimates satisfied by the rearrangement $u^{\ast}$ of a solution $u$ to a certain  and it becomes
crucial to differentiate with respect to the extra variable $y$ (introduced in the extension process that is used in fractional operators) in the form
\[
\int_{\{u(x,y)>u^{*}(s,y)\}}\frac{\partial u}{\partial y}(x,y)\,dx\,.
\]
We recall here two formulas that have  been already used in \cite{dBVol} and \cite{VazVol1}

\begin{proposition}\label{BANDLE}
Suppose that $f\in H^{1}(0,T;L^{2}(\Omega))$ for some $T>0$ and $f$ is nonnegative. Then $$f^{*}\in H^{1}(0,T;L^{2}(0,|\Omega|))$$ and if \
$|\left\{f(x,t)=f^{*}(s,t)\right\}|=0$ \
for
a.e. $(s,t)\in(0,|\Omega|)\times(0,T)$, the following differentiation formula holds:
\begin{equation}
\int_{f(x,y)>f^{*}(s,y)}\frac{\partial f}{\partial y}(x,y)\,dx=\int_{0}^{s}\frac{\partial f^{*}}{\partial y}(\tau,y)\,d\tau.\label{Rakotoson}
\end{equation}
\end{proposition}
The second-order differentiation formula is

\begin{proposition}
\label{Ferone-Mercaldo} Let $f$ nonnegative and $f\in W^{2,\infty}\left(  \mathcal{C}_{\Omega}\right)  $. Then for almost every $y\in(0,+\infty)$ the following
 differentiation formula holds:
\begin{align*}
\int_{f\left(  x,y\right)  >f^{\ast}\left(  s,y\right)  }& \frac{\partial^{2}%
f}{\partial y^{2}}\left(  x,y\right)  dx    =\frac{\partial^{2}}{\partial
y^{2}}\int_{0}^{s}f^{\ast}\left(  \tau,y\right)  d\tau-\int_{f\left(
x,y\right)  =f^{\ast}\left(  s,y\right)  }\frac{\left(  \frac{\partial
f}{\partial y}\left(  x,y\right)  \right)  ^{2}}{\left\vert \nabla
_{x}f\right\vert }\,d\mathcal{H}^{N-1}\left(  x\right) \\
&  \!\!\!+\left(  \int_{f\left(  x,y\right)  =f^{\ast}\left(  s,y\right)
}\!\frac{\frac{\partial f}{\partial y}\left(  x,y\right)  }{\left\vert
\nabla_{x}f\right\vert }\,d\mathcal{H}^{N-1}\left(  x\right)  \!\right)
^{2}\!\left(  \!\int_{f\left(  x,y\right)  =f^{\ast}\left(  s,y\right)
}\!\frac{1}{\left\vert \nabla_{x}f\right\vert }\,d\mathcal{H}^{N-1}\left(
x\right)  \!\right)  ^{-1}\!.
\end{align*}
\end{proposition}

%%%%%%%%%%%%%%%%%%%%%%%%%%%%%%%%%%%%%%%%%%%%%%%%%%%

\noindent $\bullet$  {\sl Mass concentration.}
We will provide estimates of the solutions of our elliptic and parabolic problems in terms of their integrals. For that purpose, the following
definition, taken from   \cite{Vsym82}, is remarkably useful.

\begin{definition}
Let $f,g\in L^{1}_{loc}(\R^{N})$ be two radially symmetric functions on $\R^{N}$. We say that $f$ is less concentrated than $g$, and we write
$f\prec
g$ if for
all $R>0$ we get
\[
\int_{B_{R}(0)}f(x)dx\leq \int_{B_{R}(0)}g(x)dx.
\]
\end{definition}
The partial order relationship $\prec$ is called \emph{comparison of mass concentrations}.
Of course, this definition can be suitably adapted if $f,g$ are radially symmetric and locally integrable functions on a ball $B_{R}$. Besides, if
$f$
and $g$ are locally integrable on a general open set $\Omega$, we say that $f$ is less concentrated than $g$ and we write again $f\prec g$ simply if
$f^{\#}\prec g^{\#}$,  but this extended definition has no use if $g$ is not rearranged.

The comparison of mass concentrations enjoys a nice equivalent formulation if $f$ and $g$ are rearranged, whose proof we refer to  \cite{MR0046395},
\cite{Chong}, \cite{VANS05}:

\begin{lemma}\label{lemma1}
Let $f,g\in L^{1}(\Omega)$ be two rearranged functions on a ball $\Omega=B_{R}(0)$. Then $f\prec g$ if and only if for every convex
nondecreasing
function
$\Phi:[0,\infty)\rightarrow [0,\infty)$ with $\Phi(0)=0$ we have
\begin{equation}
\int_{\Omega}\Phi(f(x))\,dx\leq \int_{\Omega}\Phi(g(x))\,dx.
\end{equation}
This result still holds if $R=\infty$ and $f,g\in L^{1}_{loc}(\R^{N})$ with $g\rightarrow0$ as $|x|\rightarrow\infty$.
\end{lemma}
From this Lemma it easily follows that if $f\prec g$ and $f,g$ are rearranged, then
\begin{equation}
\|f\|_{L^{p}(\Omega)}\leq \|g\|_{L^{p}(\Omega)}\quad \forall p\in[1,\infty].
\end{equation}

%%%%%%%%%%%%%%%%%%%%%%%%%%%%%%%%%%%%%%%%%%%%%%%%%%%%%%%%%%%%%%%%%%%%
\noindent {\bf II. On Analyticity.} In the analyticity argument of Proposition \ref{equality} we wanted to apply the results of  \cite{Hashimoto}. Let us  put
\[
F(x,z,u,u_{j},u_{jk})=z^{\nu}u_{(N+1)(N+1)}+\sum_{j=1}^{N}u_{jj}
\]
 (where subindexes indicate partial derivatives). Then, $\frac{\partial F}{\partial u_{jk}}=z^{\nu}$ for $j=k=N+1$, $\frac{\partial F}{\partial u_{jj}}=1$ for
 each $j=1,\ldots,N$ and $\frac{\partial F}{\partial
u_{jk}}=0$ for $j\neq k$, thus
\[\sum_{j,k=1}^{N+1}\frac{\partial F}{\partial u_{jk}}(x,z,u,u_{j},u_{jk})\zeta_{j}\zeta_{k}=z^{\nu}\zeta_{N+1}^{2}+|\zeta_{1}|^{2}+\ldots+|
\zeta_{N}|^{2}>0
\]
for all $(x,z)\in\R^{N+1}_{+}:=\Omega,$ $\zeta\in \R^{N+1}\setminus{(0,0)}$.
Then the equation
\[
z^{\nu}w_{zz}+\Delta_{x}w=F(x,z,\nabla_{x,z} w,\nabla_{x,z}^{2}w)=0
\]
is elliptic in $\Omega=\R^{N+1}_{+}$. Since a solution $w$ to such equation is $C^{\infty}$ in $\R^{N+1}_{+}$ and the  function $F(x,z,u,u_{j},u_{jk})$ is
analytic in $(z,u_{jj})\in\R_{+}\times\R^{N+1}$ we can apply the main Theorem in \cite{Hashimoto} and conclude that $w(x,z)$ is analytic in $\R^{N+1}_{+}$.

\medskip

%%%%%%%%%%%%%%%%%%%%%%%%%%%%%%%%%%%%%%%

\noindent {\bf III. On accretive operators and the semigroup approach}\label{sec.evolapp}
\setcounter{equation}{0}

Let $X$ be a Banach space and $\mathcal{A}:D(\mathcal{A})\subset X\rightarrow X$ a nonlinear operator defined on a suitable subset of $X$. Let us consider the
problem
\begin{equation}\label{eqcauchyabstract.3}
\left\{
\begin{array}
[c]{lll}%
u^{\prime}(t)+\mathcal{A}(u)=f,  &  & t>0,%
\\[4pt]
u(0)=u_{0}\,, &  &
\end{array}
\right.
\end{equation}
where $u_{0}\in X$ and $f\in L^{1}(I;X)$ for some interval $I$ of the real axis.  For a wide class of operators, in particular the ones considered in this
paper,
a very efficient way to approach such problem is to use an implicit time discretization scheme that we describe next. Suppose to be specific that
$I=[0,T]$ (but this can be replaced by any interval $[a,b]$ and the procedure is similar). The method
consists in taking first a partition of the interval, say, $t_k=kh$ for $k=0,1,\ldots n$ and $h=T/n$, and then solving the system of difference relations
\begin{equation}
\frac{u_{h,k}-u_{h,k-1}}{h}+\mathcal{A}(u_{h,k})=f_{k}^{(h)}\label{discrellprob}
\end{equation}
for $k=0,1,\ldots n$, where we pose $u_{h,0}=u_{0}$. The data set $\left\{f_{k}^{(h)}:k=1,\ldots,n\right\}$ is supposed to be a suitable discretization
of
the source term $f$, corresponding to the time discretization we choose. This process is called \emph{implicit time discretization scheme} (ITD for
short)
of the equation
$u^{\prime}(t)+\mathcal{A}(u)=f$. It can be rephrased in the form
\[
u_{h,k}=J_{h}(u_{h,k-1}+hf_{k}^{(h)})\,,
\]
where the operator $J_{\lambda}=(I+\lambda \mathcal{A})^{-1},\,\lambda>0
$ \ is called the \emph{resolvent operator},  $I$ being the identity operator.  Therefore, the application of the method needs the operator $\mathcal{A}$
to
have a well-defined family of resolvents with good properties.  When the ITD is solved, we construct a \emph{discrete approximate solution}
$\left\{u_{h,k}\right\}_{k}$. By piecing together the values  $u_{h,k}$ we form a piecewise constant function, $u_h(t)$,
typically defined through
\begin{equation}
u_{h}(t)=u_{h,k}\quad\text{if }t\in[(k-1)h,kh]\label{interpol}
\end{equation}
(or some other interpolation rule, like linear interpolation). Then the main question consists in verifying if such function $u_{h}$ converges
somehow
as
$h\rightarrow0$ to a solution $u$ (which we hope to be a classical, strong, weak, or other type of solution) to problem \eqref{eqcauchyabstract.3}.
To
this
regard, we first choose a suitable discretization $\left\{f_{k}^{(h)}\right\}$  in time of
the source term $f$, such that the piecewise constant interpolation of this sequence produces a function $f^{(h)}(t)$ (defined by means of
\eqref{interpol}) verifies the property
\[
\|f^{(h)}-f\|_{L^{1}(0,T;X)}\rightarrow0\quad\text{as }h\rightarrow0.
\]
By means of these discrete approximate solutions we introduce the following notion of \emph{mild solution\,}:

\begin{definition}
We say that $u\in C((0,T);X)$ is a mild solution to \eqref{eqcauchyabstract.3} if it is obtained as uniform limit of the approximate solutions $u_{h}$, as
$h\rightarrow0$. The initial data are taken in the sense that $u(t)$ is continuous in $t=0$ and $u(t)\rightarrow u_{0}$ as $t\rightarrow0$.
Besides, we say that $u\in C((0,\infty);X)$ is a mild solution to \eqref{eqcauchyabstract.3} in $[0,\infty)$ if $u$ is a mild solution to the
same problem in any compact subinterval $I\subset [0,\infty)$.
\end{definition}

In order to state a positive existence result, we  need to restrict the class of operators according to the following definitions.

\begin{definition}\label{AcRank}Let $\mathcal{A}:D(\mathcal{A})\subset X\rightarrow X$ be a nonlinear,  possibly unbounded operator. Let
$R_{\lambda}(\mathcal{A})$ be the range of $I+\lambda \mathcal{A}$, a subset of $X$.

\noindent {\rm (i)}  The operator $\mathcal{A}$ is said accretive if for all $\lambda>0$ the map $I+\lambda \mathcal{A}$ is one-to-one onto
$R_{\lambda}(\mathcal{A})\subset X$, and the resolvent operator $J_{\lambda}:R_{\lambda}(\mathcal{A})\rightarrow X$ is a (non-strict) contraction in
the
$X$-norm (i.\,e., a Lipschitz map with Lipschitz norm 1).

 \noindent {\rm  (ii)} We say that $\mathcal{A}$ satisfies the rank condition if \ $R_{\lambda}(\mathcal{A})\supset\overline{D(\mathcal{A})}$ for all
 $\lambda>0$. In particular, the rank condition is satisfied if \ $R_{\lambda}(\mathcal{A})=X$ for all $\lambda>0$; in this case, if $\mathcal{A}$ is accretive,
 we say that $\mathcal{A}$ is
 $m$-accretive.
\end{definition}

We are now ready to state the desired semigroup generation result, that generalizes the classical result of Hille-Yosida (valid in Hilbert spaces and
for
linear $\mathcal{A}$) and the variant by Lumer and Phillips (valid in Banach spaces, still for linear $\mathcal{A}$), and provides the existence and
uniqueness of mild solutions for problems of the type \eqref{eqcauchyabstract.3} in the case $f\equiv0$:

\begin{theorem}[Crandall-Liggett]\label{CrLigg} Suppose that $\mathcal{A}$ is an accretive operator satisfying the rank condition. Then for all data
$u_{0}\in \overline{D(\mathcal{A})}$ the limit
\begin{equation}
S_{t}(\mathcal{A})u_{0}=\lim_{n\rightarrow\infty}(J_{t/n}(\mathcal{A}))^{n}u_{0}.\label{CrLig}
\end{equation}
exists uniformly with respect to $t$, on compact subset of $[0,\infty)$, and  $u(t)=S_{t}(\mathcal{A})u_{0}\in C([0,\infty): X)$. Moreover, the family of
operators $\left\{S_{t}(\mathcal{A})\right\}_{t>0}$ is a strongly continuous semigroup of contractions on $\overline{D(\mathcal{A})}\subset X$.
\end{theorem}

Using a popular notation in the linear framework, we could write $S_{t}(\mathcal{A})u_{0}=e^{-t\mathcal{A}}u_{0}$, and because of this analogy formula
\eqref{CrLig} is called the \emph{Crandall-Liggett exponential formula for the nonlinear semigroup generated by} $-\mathcal{A}$. The problem with this very
general and useful result is that the $X$-valued function $u(t)=S_t(\mathcal{A})u_0$ solves the equation only in a mild sense, that is not necessarily
a strong solution or a weak solution. Though it is known that strong solutions are automatically mild, the correspondence between mild and weak solutions is not
always clear. For the FPME this issue has been discussed in detail in \cite{pqrv, pqrv2}.

In addition, the Crandall-Liggett Theorem result can be extended when we consider nontrivial source term $f$, according to the following result
 \begin{theorem}\label{existmildsol}
 Suppose that $\mathcal{A}$ is $m$-accretive. If  $f\in L^{1}(0,\infty;X)$ and $u_{0}\in\overline{D(\mathcal{A})}$. Then the abstract  problem
 \eqref{eqcauchyabstract.3} has a unique mild solution $u$, obtained as limit of the discrete approximate solution $u_{h}$ by ITD scheme described
 above,
 as $h\rightarrow0$:
 \[
 u(t):=\lim_{h\rightarrow0}u_{h}(t)\,,
 \]
 and the limit is uniform in compact subsets of $[0,\infty)$. Moreover, $u\in C([0,\infty);X)$ and for any couple of solutions $u_{1}$, $u_{2}$
 corresponding to source terms $f_1$, $f_2$ we have
 \[
 \|u_{1}(t)-u_{2}(t)\|_{X}\leq\|u_{1}(s)-u_{2}(s)\|_{X}+\int_{s}^{t}\|f_{1}(\tau)-f_{2}(\tau)\|_{X}d\tau
  \]
 for all $0\leq s<t$.
 \end{theorem}

There is a wide literature on these topics, starting with the seminal paper by Crandall and Liggett \cite{CL71}, see also \cite{Cr86} and the general
reference \cite{Barbu}.  These notes are based on Chapter 10 of the book \cite{vazquezPME}, cf. the references therein.  The last formula we have
mentioned introduces the correct concept of uniqueness for the constructed class of solutions. Characterizing the  uniqueness of different concepts of solution
is a difficult topic already discussed (with positive results) by B\'enilan in his thesis \cite{BeTh}.

%%%%%%%%%%%%%%%%%%%%%%%%%%%%%

\section{Comments and extensions}

\noindent $\bullet$ We have only proved results on  parabolic comparison based on symmetrization for the linear case. The elliptic results can be applied to
nonlinear equations but still have
severe restrictions. It will be interesting to know how much is true for nonlinear functions $B$ and $A$ in the respective equations. This question was partially
addressed and solved for the spectral fractional Laplacian in \cite{VazVol1}, and the limitations to the generality of the results were also shown to be
necessary, the symmetrization result was false for concave $B$ or convex $A$ of power type.

More generally, we would like to know if there is an approach that ensures comparison results of some symmetrization type valid for quite general nonlinearities,
as it happens in the non-fradtional case, cf. \cite{VANS05}.

\medskip

\noindent $\bullet$ {\bf The variable coefficient case.}
As a future direction, we are interested in the following problem
\begin{equation}\label{diva}
\left \{
\begin{array}{c}
\mbox{div}(\,\,
y^{1-2s}B(x)
\nabla w)=0 \quad \mbox{in}\,\, \mathcal H^+\\
w=0 \quad \mbox{in}\,\, \mathbb R^{N} \backslash \Omega \\
-y^{1-2s}\partial_y w|_{y=0}= f,
\end{array} \right .
\end{equation}
where
$$
B(x)=\begin{pmatrix}
A(x) & 0\\
0 & 1
\end{pmatrix} .
$$
Here the matrix $A(x)$ is supposed to be $W^{1,\infty}(\R^N)$ and uniformly elliptic with lower constant $\Lambda >0$.

It is a by now well-known fact that the spectral powers of $\mbox{div}(A(x)\nabla)$, i.e. $\Big ( \mbox{div}(A(x)\nabla) \Big )^s$ for $s \in (0,1)$ in a bounded
domain $\Omega $ can be described as the Dirichlet-to-Neumann operator of a suitable extension in a cylinder $\mathcal C= \Omega \times \R^+$ (see for instance
\cite{CaffaStinga} for a detailed account). The previous problem \eqref{diva} is a variant of this extension but in the whole $\R^N$. The Dirichlet-to-Neumann
operator in this case is {\sl not} explicitly identified; however, we believe it is a natural possible extension of the problem we considered in this paper. The
idea here is to develop the techniques produced in the present paper to handle variable coefficients, having in mind an isoperimetric inequality. Indeed, a FKI
was proven in terms of the first eigenvalue by means of
$$
\lambda_1(\Omega) \geq \Lambda \lambda_1(\Omega^\#).
$$
The aim here is to prove such a result for the following problem: let $\mathcal L_s$ be the Dirichlet-to-Neumann operator associated to the problem \eqref{diva}
defined on $\Omega$. It is obvious that $\mathcal L_s$ has discrete spectrum $\left \{ \lambda_{k,s} \right \}_{k=1}^\infty$. It is not clear how to use a
variational approach to deal with this operator since it does not seem obvious that this operator is associated to a norm in $\R^N$ satisfying a P\'olya-Sz\"ego
inequality. However, the parabolic approach developed in the present paper seems promising.

\

%%%%%%%%%%%%%%%%%%%%%%%%%%%%%%%%%%%%%%%%%%%%%%%%%%%%%%%%%%%%%%%%%%%
%%%%%%%%%%%%%%%%%%%%%%%%%%%%%%%%%%%%%%%%%%%%%%%%%%%%%%%%%%%%%%%%%%%

\noindent {\large\bf Acknowledgments}

\noindent   Y.S. is supported by ANR projects ``HAB'' and ``NONLOCAL''. J.L.V. partially supported by the Spanish Research Project MTM2011-24696. B.V. partially supported by the INDAM-GNAMPA project 2014 ``\textit{Analisi qualitativa di soluzioni di equazioni ellittiche e di
evoluzione}'' (ITALY).
\normalcolor

\
%newpage

%%%%%%%%%%%%%%%%%%%%%%%%%%%%%%%%%%%%%%%%%%%%%%%

{\small

%%%%%%%%%%%%%%%%%%%%

}

\

2000 \textit{Mathematics Subject Classification.}
35B45,  % 	A priori estimates
35R11,   	% Fractional partial differential equations
35J61, %  	Semilinear elliptic equations
35K55. %Nonlinear PDE of parabolic type

%%%%%%%%%%%%%%%%%%%%%%%%%%%%%%%%%%%%%%%%%%%%%%%%%%%%%%%%%%%%%%%%%

%
\textit{Keywords and phrases.} Symmetrization, fractional Laplacian,
 nonlocal elliptic and parabolic equations, Faber-Krahn inequality.

\end{document}